%% file: mencurv.tex
\documentclass[10pt,english,a4paper,intlimits]{scrartcl}

\input{preambule.tex}
\input{macros.tex}

\date{\today}

\title{Sharp Boundedness and Regularizing effects of the integral Menger
  curvature for submanifolds}

\author{Simon Blatt \footnote{Mathematics Institute,
    Zeeman Building,
    University of Warwick,
    Coventry CV4 7AL,
    United Kingdom,
    S.Blatt@warwick.ac.uk},
  S{\l}awomir Kolasi{\'n}ski\footnote{Institute of Mathematics,
    University of Warsaw,
    Banacha 2, 02-097 Warsaw,
    Poland,
    s.kolasinski@mimuw.edu.pl}}

\begin{document}

\maketitle

\begin{abstract}
  In this paper we show that embedded and compact $C^1$ manifolds have finite
  integral Menger curvature if and only if they are locally graphs of certain
  Sobolev-Slobodeckij spaces. Furthermore, we prove that for some intermediate
  energies of integral Menger type a similar characterization of objects with
  finite energy can be given.
\end{abstract}

\tableofcontents

\section{Introduction}

To study the geometry of metric spaces, Karl Menger found a way to define the
curvature of a curve without using any parameterization of this geometric
object~\cite{Menger1930}. For each triple of points $(x,y,z)$ lying on the
curve he looked at the reciprocal of the radius of the circumscribing circle
of the three points. This quantity is nowadays named ``Menger curvature'' and
will be denoted by $c(x,y,z)$ in this article.  Menger observed that one
obtains the curvature of the curve at a point $p$ by the limit of the this
Menger curvature $(x,y,z)$ as the three points converge to $p$.

The growing interest in this quantity during the last years started with the
observation that Menger curvature has a tight relation to many modern fields
in mathematics apart from metric geometry. A milestone is certainly the
discovery of the intimate relation between total Menger curvature of
an~$\HM^1$ measurable set $K$ - given by
\begin{equation*}
  \M_2 (K) := \int_{K}\int_{K} \int_{K} c^2(x,y,z)\ d\HM^1_x\ d\HM^1_y\ d\HM^1_z \quad \text{-}
\end{equation*}
and harmonic analysis, rectifiability, and analytic capacity
(see~\cite{MatSurv98} or~\cite{TolSurv06}). M.~Leger proved
in~\cite{Leger1999} that finite global Menger curvature implies that the set
is rectifiable. Using this result, Guy David proved that $\M_2(K) < \infty$ is a
sufficient condition for a set $K$ to have vanishing analytic capacity. This
enabled him to prove the Vitushkin's conjecture~\cite{David1998} for sets of
finite one-dimensional Hausdorff measure.

Another application of Menger curvature is its use as basic building block in
the construction of so called ``knot energies'' - energies that penalize self
intersections and thus can hopefully be minimized within a given knot class.
These energies play an important role in the modeling of the structure of polymer chains like
proteins and DNA.

The first to use Menger curvature to define such self-repulsive energies were Oscar Gonzales and John H. Maddocks.
In~\cite{Gonzalez1999}, they introduced and analyzed the notion of
\emph{global radius of curvature} of a curve $\gamma$ given by
\begin{equation*}
  \rho (\gamma):=\inf_{x,y,z \in \gamma(\R / \Z)} \frac 1 {c(x,y,z)}.
\end{equation*}
At the end of this article, they also suggest the investigation of the
integral versions
\begin{align*}
  \mathcal U_p (\gamma) &:= \int_{\R / \Z} \sup_{y,z \in \R / \Z } c^p(\gamma(x), \gamma(y),\gamma(z)) |\gamma'(x)|
  \ dx, \\
  \mathcal I_p (\gamma) &:= \int_{\R / \Z} \int_{\R / \Z} \sup_{z \in \R / \Z } c^p(\gamma(x), \gamma(y),\gamma(z)) 
  |\gamma'(x)| |\gamma'(y)| \,
  \ dx\ dy, \\
  \intertext{and}
  \mathcal M_p (\gamma) &:= \int_{\R / \Z} \int_{\R / \Z} \int_{ \R / \Z } c^p(\gamma(x), \gamma(y),\gamma(z)) 
  |\gamma'(x)| |\gamma'(y)| |\gamma'(z)|\,
  \ dx\ dy\ dz,
\end{align*}
- a program that was pushed forward in a series of groundbreaking papers by
Pawe{\l} Strzelecki, Heiko von der Mosel and Marta
Szuma{\'n}ska~\cite{Strzelecki2007, Strzelecki2010} in which they could show,
apart from other things, that for suitable $p$ these energies are self repulsive
and possess certain regularizing effects - and thus are indeed worth being
called ``knot energies''.

Generalizing the notion of Menger curvature from one-dimensional to higher
dimensional objects, was not a trivial task. The obvious generalization --
i.e. taking the inverse of the radius of an $m$-dimensional sphere defined by
its $m+2$ points -- seems not to be the right Ansatz from an analytic point of
view.  Strzelecki and von der Mosel have given examples
(see~\cite[Appendix~B]{0911.2095}) of smooth embedded manifolds for which the
resulting integral curvatures are unbounded. But more promising candidates were
introduce and successfully investigated in \cite{0805.1425, MR2558685,
  0911.2095} and \cite{Kolasinski2011}.

In this article we will look at the variant of integral Menger curvature for
submanifolds of the Euclidean space of arbitrary dimension and codimension
introduced in~\cite{Kolasinski2011} -- and sometimes laxly refer to it as Menger
curvature being aware that there are other quantities that deserve this name.

Motivated by the formula
\begin{equation*}
  c(x,y,z) = 4 \frac{\mathcal H^2(\simp(x,y,z))}{|x-y||y-z||z-x|},
\end{equation*}
where $\simp(x,y,z)$ stands for the convex hull of the points $x,y$, and $z$, we
are led to use the quantity
\begin{displaymath}
  \DC(x_0, \ldots, x_{m+1}) = \frac{\HM^{m+1}(\simp(x_0, \ldots, x_{m+1}))}
  {(\diam \{x_0, \ldots, x_{m+1}\})^{m+2}}
\end{displaymath}
as a substitute for the Menger curvature of curves. Here again $\simp(x_0,
\ldots,x_{m+1})$ stands for the convex hull of the points $x_0$, \ldots,
$x_{m+1}$ in $\R^n$. We take the diameter of the set of points $\{x_0, \ldots,
x_{m+1}\}$ to the power $m+2$ in the denominator, which guarantees that this
quantity scales like a curvature.

It is easy to check that for triples $(x,y,z)$ we always have $4 \DC(x,y,z) \le
c(x,y,z)$ and that for a class of triangles with comparable sides, (i.e. $|x-y|
\simeq |y-z| \simeq |z-x|$) the two quantities $\DC(x,y,z)$ and $c(x,y,z)$ are
comparable. It is also obvious that for general triangles this is not true.

Following the suggestion of Gonzalez and Maddocks mentioned above, one is led to
the following \emph{intermediate integral Menger curvatures}
\begin{displaymath}
  \E_p^k(\Sigma) = \int_{\Sigma^k} \sup_{x_k, \ldots x_{m+1} \in \Sigma}
  \DC(x_0,\ldots, x_{k-1})^p\ d \HM^{mk}_{x_0, \ldots, x_{k-1}} \,.
\end{displaymath}
for $k \in \{1, \ldots, m+1\}$ and the \emph{integral Menger curvature}
\begin{displaymath}
  \E_p = \E_p^{m+2}(\Sigma) = \int_{\Sigma^{m+2}}
  \DC(x_0,\ldots, x_{m+1})^p\ d \HM^{m(m+2)}_{x_0, \ldots, x_{m+1}} \,
\end{displaymath}
discussed in~\cite{Kolasinski2011}.

The main result of this article is the following characterization of all compact
embedded $C^1$ submanifolds with finite energy $\E^k_p$ for $k \in \{2, \ldots,
m+1\}$.
\begin{thm}
  \label{thm:MainTheorem}
  Let $m,n,k \in \N$, $p \in \R$ satisfy $m < n$, $2 \le k \le m+2$ and $p >
  m(k-1)$. Furthermore, let $\Sigma \subseteq \R^n$ be a compact
  $m$-dimensional $C^1$ manifold and $s = 1 - \frac{m(k-1)}p \in (0,1)$. Then
  $\E_p^k(\Sigma)$ is finite if and only if $\Sigma$ can locally be
  represented as the graph of a~function belonging to the Sobolev-Slobodeckij
  space $W^{1+s,p}(\R^m,\R^{n-m})$.
\end{thm}

Here $W^{s,p}$ stands for the Sobolev-Slobodeckij spaces. For a definition of
these spaces, some basic properties, and references see
Section~\ref{sec:SlobSpaces}.

Note, that a classification of all finite energy objects for $\E^1_p$ for $p \in
[1,\infty]$ was already achieved in~\cite{KStvdM2011} and \cite{Gonzalez1999} --
essentially these are the embedded $W^{2,p}$ submanifolds. So we now have a
complete classification of $C^1$ manifolds with finite energy for all
intermediate integral Menger curvatures.

As an immediate consequence of Theorem~\ref{thm:MainTheorem} and the results
in~\cite{Kolasinski2011}, one gets for the integral Menger curvature

\begin{cor}
  \label{cor:AdmissibleSets}
  Let $m,n \in \N$, $p \in \R$ satisfy $m < n$ and $p > m(m+2)$. Furthermore,
  let $\Sigma$ be an admissible compact set in the sense
  of~\cite{Kolasinski2011}. Then $\E_p(\Sigma)$ is finite if and only if
  $\Sigma$ can locally be represented as the graph of some function belonging to
  the Sobolev-Slobodeckij space $W^{1+s,p}(\R^m,\R^{n-m})$, where $s = 1 -
  \frac{m(m+1)}p \in (0,1)$.
\end{cor}

For curves, the classification of finite energy objects for $\mathcal M_p$,
$\mathcal I_p$ was achieved in \cite{Strzelecki2007,Bl-note}. It is a surprising
fact, that though $\E^p_3$, $\E^p_2$, $\E^p_1$ for curves look much weaker than
$\mathcal M_p$, $\mathcal I_p$, and $\mathcal U_p$, the corresponding energies
are bounded on exactly the same objects.

In ~\cite{KolSzum11} the optimal H\"older regularity that implies finiteness of
$\mathcal M_p$ or $\E_p$ was deduced. Note, that this result in any dimension
and for all intermediate energies can now be interpreted as a simple
consequences of Theorem~\ref{thm:MainTheorem} and classical embedding and
non-embedding theorems of Sobolev-Slobodeckij spaces.

\section{Sobolev-Slobodeckij spaces} \label{sec:SlobSpaces}

For the readers convenience we repeat some well known facts about Sobolev-Slobodeckij spaces.

\begin{defin}[\!\!\cite{Adams75}, Chapter VII]
  \label{def:SS-space}
  Let $k \in \N$, $s \in (0,1)$, $p \ge 1$ and let $\Omega$ be an open subset of
  $\R^m$ with smooth boundary. We say that $u \in L^p(\Omega)$ belongs to the
  \emph{Sobolev-Slobodeckij space} $W^{k+s,p}(\Omega)$ if
  \begin{displaymath}
    \|u\|_{W^{k+s,p}(\Omega)} = \left(
      \|u\|_{W^{k,p}(\Omega)}^p + \sum_{|\alpha| = k} \int_{\Omega} \int_{\Omega}
      \frac{|D^{\alpha}u(x) - D^{\alpha}u(y)|^p}{|x-y|^{m+sp}\ dy\ dx}
    \right)^{\frac 1p} < \infty \,.
  \end{displaymath}
\end{defin}

When we show that boundedness of $\E^{k}_p$ implies that the submanifold was of class 
$W^{1+s,p}$, we will use a different but equivalent norm on these spaces due to Besov:

\begin{defin}[\!\!\cite{Adams75}, 7.67]
  \label{def:besov-norms}
  Let $k \in \N$, $s \in (0,1)$, $p \ge 1$ and let $\Omega$ be an open subset of
  $\R^m$ with smooth boundary. For $x \in \Omega$, we set
  \begin{displaymath}
    \Omega_x = \{ y \in \Omega : \tfrac 12 (x+y) \in \Omega \} \,.
  \end{displaymath}
  For $u \in W^{k,p}(\Omega)$ we say that $u \in B^{k+s,p}(\Omega)$ if
  \begin{displaymath}
    \|u\|_{B^{k+s,p}(\Omega)} = \left(
      \|u\|_{W^{k,p}(\Omega)}^p + \sum_{|\alpha| = k} \int_{\Omega} \int_{\Omega_x} 
      \frac{|D^{\alpha}u(x) - 2D^{\alpha}u(\tfrac 12 (x+y)) + D^{\alpha}u(y)|^p}{|x-y|^{m+sp}} \ dy\ dx
    \right)^{\frac 1p}
  \end{displaymath}
  is finite.
\end{defin}

\begin{thm} [cf. \cite{Triebel1995}, Theorem~2.5.1]
  \label{thm:equiv-norms}
  Let $k \in \N$, $s \in (0,1)$, $p \ge 1$ and let $\Omega$ be an open subset of
  $\R^m$ with smooth boundary.  Then we have $W^{k+s,p}(\Omega)=
  B^{k+s,p}(\Omega)$ and the norms $\|\cdot\|_{W^{k+s,p}(\Omega)}$ and
  $\|\cdot\|_{B^{k+s,p}(\Omega)}$ are equivalent. Moreover, for $\sigma \in
  (0,2)$ the norm $\|\cdot\|_{B^{\sigma,p}(\Omega)}$ is also equivalent to the
  following norm
  \begin{displaymath}
    \|u\| = \left(
      \|u\|_{L^p(\Omega)}^p + \int_{\Omega} \int_{\Omega_x} 
      \frac{|u(x) - 2u(\tfrac 12 (x+y)) + u(y)|^p}{|x-y|^{m+\sigma p}} \ dy\ dx
    \right)^{\frac 1p} < \infty \,.
  \end{displaymath}
\end{thm}

Furthermore, we will use the following extension lemma

\begin{thm}[cf. \!\!\cite{Triebel1983}, Theorem p. 201] \label{thm:extension}
  Let $k \in \N$, $s \in (0,1)$, $p \ge 1$ and let $\Omega$ be an open subset of
  $\R^m$ with smooth boundary. Then $u \in W^{k+s,p}(\Omega)$ if and 
  only if it is the restriction of a function $ \tilde u \in W^{k+s,p}(\mathbb R^m)$
  onto $\Omega$. 
\end{thm}

Apart from this, we will need the following
well known embedding theorem

\begin{thm}[\!\!\cite{Adams75}, Theorem~7.57]
  \label{thm:SS-embed}
  Let $s > 0$ and $1 < p < m$. If $n
  < (s-j)p$ for some nonnegative integer $j$, then $W^{s,p}(\Omega) \subset
  C^j_{loc}(\mathbb R^m)$. 
\end{thm}

\section{Being a $W^{1+s,p}$ submanifold implies $\E_p^k < \infty$}

Since we will have to work with balls of different dimensions in this article,
let us introduce the 
symbol $B^l(r,x)$ for the $l$-dimensional open ball in $\mathbb R^l$.

In this section we are proving the following half of our main theorem
\begin{thm}
  \label{thm:Boundedness}
  Fix some natural number $2 \le k \le m+2$. Let $p > m(k-1)$ and $s = 1 -
  \frac{m(k-1)}p \in (0,1)$. Let $\Sigma \subseteq \R^n$ be a compact
  $m$-dimensional manifold, with local graph representation in the
  Sobolev-Slobodeckij space $W^{1+s,p}(\R^m,\R^{n-m})$. Then $\E_p^k(\Sigma)$
  is finite.
\end{thm}

Throughout this section we use the symbol $T_k$ to denote a $k$-tuple
$T=(w_0,\ldots,w_{k-1})$ of $k$ points in $\R^n$. Using this notation we can
write
\begin{displaymath}
  \DC(T_{m+2}) = \frac{\HM^{m+1}(\simp T_{m+2})}{(\diam T_{m+2})^{m+2}} \,.
\end{displaymath}
We define the measure
\begin{displaymath}
  \mu_k = \underbrace{\HM^m \otimes \cdots \otimes \HM^m}_{k \text{ times}}
\end{displaymath}
and set
\begin{displaymath}
  \DC_k(x_0,\ldots,x_{k-1}) = \sup_{x_k,\ldots,x_{m+1} \in \Sigma} \DC(x_0,\ldots,x_{m+1})
\end{displaymath}
for $k\in\{1,\ldots, m+1\}$ and
\begin{displaymath}
  \DC_{m+2}(x_0,\ldots,x_{k-1}) = \DC(x_0,\ldots,x_{m+1}) \,.
\end{displaymath}
Now, we can write
\begin{displaymath}
  \E_p^k(\Sigma) = \int_{\Sigma^k} \DC_k(T_k)^p\ d\mu_k(T_k) \,
\end{displaymath}
for all $k \in \{1, m+2\}.$

For any set $A \subseteq \R^n$ and any $\lambda > 0$ we define
\begin{align*}
  A^k_{\ge \lambda} &= \{(w_1,\ldots,w_k) \in A^k : \diam\{w_1,\ldots,w_k\} \ge \lambda \} \\
  \text{and} \quad
  A^k_{< \lambda} &= \{(w_1,\ldots,w_k) \in A^k : \diam\{w_1,\ldots,w_k\} < \lambda \} = A^k \setminus A^k_{\ge \lambda} \,.
\end{align*}

Let $o \in \Sigma$ and let $\rho > \lambda > 0$ be some numbers. We set
\begin{displaymath}
  \Sigma_o^{\rho} = \Sigma \cap \Ball(o,\rho)
  \quad \text{and} \quad
  \DC_{k,o,\rho}(w_0,\ldots,w_{k-1}) = \sup_{w_k,\ldots,w_{m+1} \in \Sigma_o^{\rho}} \DC(w_0,\ldots,w_{m+1})
\end{displaymath}
and introduce the local version of our energy
\begin{displaymath}
  \E_p^k(\Sigma,\rho,\lambda,o) = \int_{(\Sigma_o^{\rho})^k_{< \lambda}} \DC_{k,o,2\rho}(T_k)^p\ d\mu_k(T_k) \,.
\end{displaymath}

The proof of Theorem~\ref{thm:Boundedness} relies on the following two
lemmata, the proof of which we will postpone till the end of this section.
The first one tells us, that we only have to consider simplices with small
diameter.

\begin{lem}
  \label{lem:patches}
  For any $\rho > 0$ there exist $\lambda \in (0,\rho)$, $N \in \N$, an
  $N$-tuple of points $x_1$,\ldots,$x_N$ in $\Sigma$ and a constant $C =
  C(n,m)$ such that
  \begin{displaymath}
    \E_p^k(\Sigma) \le C(n,m) \HM^m(\Sigma)^k (\lambda^{-p} + \rho^{-p})
    + \sum_{i=1}^N \E_p^k(\Sigma,\rho,\lambda,x_i)
  \end{displaymath}
\end{lem}

The second lemma tells us that in order to prove Theorem~\ref{thm:Boundedness}
it is enough to get some good estimates for the~Jones' $\beta$-numbers. 
Those are given by

\begin{defin}
  \label{def:beta-num}
  For $x \in \Sigma$ and $r>0$ we define the \emph{Jones' $\beta$-numbers}
  \begin{equation*}
    \beta(x,r):= \inf \left\{
      \frac{\sup_{y \in \Sigma \cap \Ball_(x,r)}\dist(y,H)}{r} : H \text{ an affine }m
      \text{-dimensional space containing }x
    \right\} \,.
  \end{equation*}
\end{defin}

\begin{lem}
  \label{lem:dc-beta}
  There exists a constant $C = C(m,n)$ such that for all $\Sigma \subseteq
  \R^n$ and $T_{m+2}=(x_0,\ldots,x_{m+1}) \in \Sigma^{m+2}$ we have
  \begin{displaymath}
    \HM^{m+1}(\simp T_{m+2}) \le C \beta(x_0,\diam(T_{m+2})) \diam(T_{m+2})^{m+1} 
  \end{displaymath}
  and consequently
  \begin{displaymath}
    \DC(T_{m+2}) \le C \frac{\beta(x_0,\diam(T_{m+2}))}{\diam(T_{m+2})} \,.
  \end{displaymath}
\end{lem}

In fact, we will only use the following immediate consequence of
Lemma~\ref{lem:dc-beta}
\begin{cor}
  \label{cor:dck-beta}
  For $T_k = (x_0,\ldots,x_{k-1}) \in (\Sigma_o^{\rho})^k$ we have
  \begin{displaymath}
    \DC_{k,o,2\rho}(T_k)
    \le C \sup_{x_k,\ldots,x_{m+1} \in \Sigma_o^{2\rho}}
    \frac{\beta(x_0,\diam(x_0,\ldots,x_{m+1}))}{\diam(x_0,\ldots,x_{m+1})}
    \le C \sup_{\diam(T_k) \le r \le 4\rho} \frac{\beta(x_0,r)}r \notag \,.
  \end{displaymath}
\end{cor}

Let us now show how these lemmata can be used to prove
Theorem~\ref{thm:Boundedness}:
\begin{proof}[Proof of Theorem~\ref{thm:Boundedness}]
  Despite the fact that the integrand $\DC_k(x_0,\ldots,x_{k-1})$ depends on the
  whole of $\Sigma$, Lemma~\ref{lem:patches} tells us that it is enough to show
  that there is a $\rho > 0$ such $\E_p^k(\Sigma,\lambda,\rho,o)$ is finite for
  every $\lambda \in (0,\rho)$ and every $o \in \Sigma$. The Sobolev embedding
  theorem (Theorem~\ref{thm:SS-embed}) shows that $\Sigma$ is a compact $C^1$
  submanifold of $\R^n$. Together with the fact that $\Sigma$ is locally the
  graph of a $W^{1+s,p}$ function, this allow us to choose $\rho > 0$ so~small
  that for all $o \in \Sigma$ we~have after a suitable rotation of the ambient
  space
  \begin{equation}
    \label{graph-repr}
    (\Sigma - o) \cap \Ball_{10\rho} \subseteq \graph(f) = \big\{ (x,f(x)) \in \R^n : x \in \R^m \big\} \,,
  \end{equation}
  for some function $f \in W^{1+s,p}(\R^m, \R^{n-m})$ (depending on the choice
  of $o \in \Sigma$) that satisfies
  \begin{equation}
    \label{est:lip-const}
    \forall x,y \in \Disc_{10\rho} \quad |f(x) - f(y)| \le |x-y| \,.
  \end{equation}

  Let $0 \le i < j \le k-1$ and let $u,v \in \Sigma_o^{\rho}$. We set
  \begin{align*}
    \Sigma_{i,j} &= \left\{
      (w_0,\ldots,w_{k-1}) \in (\Sigma_o^{\rho})^k : \diam\{w_0,\ldots,w_{k-1}\} = |w_i - w_j|
    \right\} \\
    \text{and} \quad
    \Sigma(u,v) &= \left\{
      (w_1,\ldots,w_{k-2}) \in (\Sigma_o^{\rho})^{k-2} : \diam\{u,w_1,\ldots,w_{k-2},v\} = |v-u|
    \right\} \,.
  \end{align*}
  For any $(w_1,\ldots,w_{k-2}) \in \Sigma(u,v)$ and $j = 1,\ldots,k-2$, we
  have $|w_j-u| \le |v-u|$. Hence,
  \begin{displaymath}
    \HM^{m(k-2)}(\Sigma(u,v))
    \le \Big( 2^m \omega_m |v-u|^m \Big)^{k-2}
    \le C(m,k) |v-u|^{m(k-2)} \,,
  \end{displaymath}
  where $\omega_m$ denotes the volume of the $m$-dimensional unit ball.
  Note that
  \begin{displaymath}
    (\Sigma_o^{\rho})^k = \bigcup \big\{ \Sigma_{i,j} : 0 \le i < j \le k-1 \big\}\,.
  \end{displaymath}
  Since $\DC$ is invariant under permutations of its parameters, so is
  $\DC_{k,o,2\rho}$ and we have
  \begin{displaymath}
    \int_{\Sigma_{i,j}} \DC_{k,o,2\rho}(T_k)^p\ d\mu_k(T_k) = \int_{\Sigma_{a,b}} \DC_{k,o,2\rho}(T_k)^p\ d\mu_k(T_k) \,,
  \end{displaymath}
  for any $i < j$ and $a < b$. Hence
  \begin{multline}
    \label{est:Ep}
    \E_p^k(\Sigma,\rho,\lambda,o)
    = \int_{(\Sigma_o^{\rho})^k_{<\lambda}} \DC_{k,o,2\rho}(T_k)^p\ d\mu_k(T_k) 
    \le \sum_{0 \le i < j \le k-1} \int_{\Sigma_{i,j} \cap \{ |w_i - w_j| < \lambda \}} \DC_{k,o,2\rho}(T_k)^p\ d\mu_k(T_k) \\
    = 2 \binom k2
    \int_{\Sigma_o^{\rho}}
    \int_{\Sigma_o^{\rho} \cap \Ball(u,\lambda)}
    \int_{\Sigma(u,v)} \DC_{k,o,2\rho}(u,w_1,\ldots,w_{k-2},v)^p\
    d\HM^{m(k-2)}_{w_1,\ldots,w_{k-2}}\ d\HM^m_v\ d\HM^m_u \,.
  \end{multline}
  Let $|JF(z)| = \sqrt{\det(DF(z)^* DF(z))}$) denote the Jacobian of $F(z) =
  (z,f(z))$. Set $\beta_o(x,r) := \beta(o+x,r)$. Using Lemma~\ref{lem:dc-beta}
  and Corollary~\ref{cor:dck-beta} we may write
  \begin{multline}
    \label{Ep-beta}
    \E_p^k(\Sigma,\rho,\lambda,o)
    \le C \int_{\Sigma_o^{\rho}} \int_{\Sigma_o^{\rho} \cap \Ball(u,\lambda)}
    |v-u|^{m(k-2)} \sup_{|u-v| \le r \le 4\rho} \frac{\beta(u,r)^p}{r^p}\ d\HM^m_v\ d\HM^m_u \\
    \le C \int_{\Disc_{\rho}} \int_{\Disc(x,\lambda)}
    |F(y)-F(x)|^{m(k-2)}
    \sup_{|F(y)-F(x)| \le r \le 4\rho}
    \frac{\beta_o(F(x),r)^p}{r^p}
    |JF(x)| |JF(y)|\ dy\ dx \\
    \le C' \int_{\Disc_{\rho}} \int_{\Disc(x,\lambda)}
    |y-x|^{m(k-2)}
    \sup_{|y-x| \le r \le 4\rho}
    \frac{\beta_o(F(x),r)^p}{r^p}\ dy\ dx \,.
  \end{multline}
  To get to the last line, we used the fact that $F$ satisfies
  (cf.~\eqref{est:lip-const})
  \begin{displaymath}
    |y-x| \le |F(y) - F(x)| \le 2 |y-x| \,,
    \quad \text{hence also} \quad
    |JF(z)| \le 2^m \,.
  \end{displaymath}
  We set 
  \begin{displaymath}
    \Sigma_o^{x,r} = (\Sigma - o) \cap \Ball(F(x),r) \,.
  \end{displaymath}
  Observe that $r \beta(u,r)$ can be estimated by the distance of $\Sigma \cap
  \Ball(u,r)$ from the affine tangent plane $u+T_u\Sigma$. Hence, recalling the
  definition of the $\beta$-numbers, we~get
  \begin{align}
    \beta_o(F(x),r)
    &\le r^{-1} \inf_{H \in G(n,m)} \sup \big\{ \dist(w, F(x)+H) : w \in \Sigma_o^{x,r} \big\} \notag\\
    &\le r^{-1} \sup \big\{ \dist(w, F(x)+T_{F(x)}(\Sigma-o)) : w \in \Sigma_o^{x,r} \big\} \notag \\
    &\le r^{-1} \sup \big\{ |F(z) - F(x) - DF(x)(z-x)| : z \in \Disc(x,2r) \big\} \notag \\
    \label{est:beta-uv}
    &= r^{-1} \sup \big\{ |f(z) - f(x) - Df(x)(z-x)| : z \in \Disc(x,2r) \big\} \,. 
  \end{align}
  Plugging \eqref{est:beta-uv} into \eqref{Ep-beta}, we are led to
  \begin{displaymath}
    \E_p^k(\Sigma,\rho,\lambda,o) 
    \le C \int_{\Disc_{\rho}} \int_{\Disc(x,\lambda)}
    |y-x|^{m(k-2)}
    \sup_{\substack{|y-x| \le r \le 4\rho\\\substack{z \in \Disc(x,2r)}}}
    \frac{|f(z) - f(x) - Df(x)(z-x)|^p}{r^{2p}}
    \ dy\ dx \,.
  \end{displaymath}

  To estimate the term $|f(z) - f(x) - Df(x)(z-x)|$ we set
  \begin{displaymath}
    g_x(z) = f(z) - f(x) - Df(x)(z-x) \,, \quad \text{then } g_x(x) = 0 \,.
  \end{displaymath}
  Since $f \in W^{1+s,p} \subseteq W^{1,p}$ and $p>m$, using the
  Sobolev-Morrey embedding theorem, we obtain
  \begin{multline*}
    \sup_{z \in \Disc(x,2r)} |g_x(z) - g_x(x)| 
    \le C \sup_{z \in \Disc(x,2r)} |z-x|^{1-\frac mp} \left(
      \int_{\Disc(\frac 12(z+x),|z-x|)} |Dg_x(t)|^p\ dt
    \right)^{\frac 1p} \\
    \le \widetilde{C} r^{1-\frac mp} \left(
      \int_{\Disc(x,5r)} |Dg_x(t)|^p\ dt
    \right)^{\frac 1p} 
    = \widehat{C} r^{1- \frac mp} \left(
      \int_{\Disc(x,5r)} |Df(t) - Df(x)|^p\ dt
    \right)^{\frac 1p} \,,
  \end{multline*}
  where the right hand side does not depend on $z$ anymore. Hence, we get the
  estimate
  \begin{equation}
    \label{est:Epk}
    \E_p^k(\Sigma,\rho,\lambda,o) 
    \le C \int_{\Disc_{\rho}} \int_{\Disc(x,\lambda)}
    |y-x|^{m(k-2)}
    \sup_{r \ge |y-x|} \frac{\left(\int_{\Disc(x,5r)} |Df(t) - Df(x)|^p\ dt\right)} {r^{m+p}}
    \ dy\ dx \,.
  \end{equation}
  Observe that
  \begin{align}
    \sup_{r \ge |y-x|} \frac {\int_{\Disc(x,5r)} |Df(t) - Df(x)|^p\ dt}{r^{m+p}}
    &\le C \sup_{r \ge |y-x|} \int_r^{2r}
    \frac{\int_{\Disc(x,5r)} |Df(t) - Df(x)|^p\ dt}{\tau^{m+p+1}} \ d\tau \notag \\
    &\le C \int_{|y-x|}^\infty  \frac{\int_{\Disc(x,5\tau)} |Df(t) - Df(x)|^p\ dt}{\tau^{m+p+1}}\ d\tau \notag \\
    &\le C \int_{\R^m} \int_{\tau \ge \max\{|t-x|/5,|y-x|\}}
    \frac {|Df(t) - Df(x)|^p}{\tau^{m+p+1}}\ d\tau\ dt \notag \\
    \label{est:rxy}
    &\le C \int_{\R^m} \frac{|Df(t) - Df(x)|^p }{\max\{|t-x|/5,|y-x|\}^{m+p}}\ dt \,.
  \end{align}
  Note that for any choice of $x$, $y$ and $t$ we have
  \begin{equation}
    \label{est:yxmax}
    \frac{|y-x|^{m(k-2)}}{\max\{|t-x|,|y-x|\}^{m(k-2)}} \le 1 \,.
  \end{equation}
  Combining~\eqref{est:rxy} and~\eqref{est:yxmax} with~\eqref{est:Epk} and using
  Fubini's theorem, we are led to
  \begin{align}
    \E^k_p(\Sigma, \rho, \lambda, o) 
    &\le C \int_{\R^m} \int_{\R^m} \int_{\R^m} 
    \frac{|Df(t) - Df(x)|^p }{\max\{|t-x|,|y-x|\}^{m+p-m(k-2)}} \ dt\ dy\ dx \notag \\
    \label{est:Epk2}
    &\le C \int_{\R^m} \int_{\R^m} \int_{\R^m}
    \frac{|Df(t) - Df(x)|^p}{\max\{|t-x|,|y-x|\}^{m+p-m(k-2)}}\ dy\ dt\ dx \,.
  \end{align}
  We can compute the innermost integral by dividing it into two parts
  \begin{multline}
    \label{est:inner}
    \int_{\R^m} \frac{|Df(t) - Df(x)|^p}{\max\{|t-x|,|y-x|\}^{m+p-m(k-2)}}\ dy
    = \int_{|y-x| \le |t-x|} \frac{|Df(t) - Df(x)|^p}{|t-x|^{m+p-m(k-2)}}\ dy \\
    + \int_{|y-x| > |t-x|} \frac{|Df(t) - Df(x)|^p}{|y-x|^{m+p-m(k-2)}}\ dy
    = C \frac{|Df(t) - Df(x)|^p}{|t-x|^{p-m(k-2)}} \,.
  \end{multline}
  Plugging~\eqref{est:inner} into~\eqref{est:Epk2} we finally get
  \begin{displaymath}
    \E^k_p(\Sigma, \rho, \lambda, o) 
    \le C \int_{\R^m} \int_{\R^m} \frac{|Df(t)-Df(x)|^p}{|t-x|^{p-m(k-2)}}\ dt\ dx
    \le C \|Df\|_{W^{s,p}} \,,
  \end{displaymath}
  where $s= 1 - \frac {m(k-1)}{p}$.
\end{proof}

\begin{proof}[Proof of Lemma~\ref{lem:dc-beta}]
  For a linear subspace $W$ of $\R^n$ let $P_W$ denote the orthogonal
  projection of $\R^n$ onto $W$ and $P^\bot_W:= id_{\R^n} - P_W$ be the
  orthogonal projection of $\R^n$ onto the orthogonal complement of $V$.
  Furthermore, let $\mathbf{T} = \simp T_{m+2}$ and let $d =
  \diam(\mathbf{T})$.

  Without loss of generality we can assume that $x_0=0$. If the vectors
  $\{x_1, \ldots, x_{m+1}\}$ are not linearly independent, then
  $\HM^{m+1}(\mathbf{T}) = 0$ and the statement of the lemma is true.

  Let $x_1, \ldots x_{m+1}$ be linearly independent and let $W$ denote the
  $(m+1)$-dimensional vector space spanned be these vectors. Set
  \begin{equation*}
    \mathbf{S} := \{ s \in W^\bot : |s| \le \beta(0,d)d \} \,.
  \end{equation*}
  Then, for the set $\mathbf{T} + \mathbf{S}$, using Fubini's theorem we obtain
  \begin{equation}
    \label{eq:VolumeST}
    \HM^{n}(\mathbf{T}+\mathbf{S}) 
    = \HM^{m+1}(\mathbf{T}) \HM^{n-m-1}(\mathbf{S})
    = \omega_m \HM^{m+1}(\mathbf{T}) d^{n-m-1} \beta(0,d)^{n-m-1}
  \end{equation}
  where $\omega_m$ is the volume of the $m$-dimensional unit ball.

  From the definition of the Jones' $\beta$-numbers we can find a sequence of
  $m$-dimensional vector spaces $V_j$ such that
  \begin{equation*}
    \sup_{y \in \Sigma \cap \Ball(x_0,d)} |P_{V_j}^\bot(y)| \le \big( \beta(x_0,d) + \tfrac 1j \big)d \,.
  \end{equation*}
  Since the Grassmannian $G(n,m)$ of all $m$-dimensional subspaces of $\R^n$ is
  a compact manifold, we can find a subsequence $V_{j_k}$ converging to some $V
  \in G(n,m)$. Observe also that the mapping $Q : G(n,m) \to \R^n$ given by
  $Q(V) = P_V(y)$ is continuous\footnote{The metric on $G(n,m)$ is defined by
    the formula $\dist(U,V)=\|P_U - P_V\|$.} for any choice of $y \in \R^n$.
  In~consequence, we get the estimate
  \begin{equation*}
    \forall y \in \Sigma \cap \Ball(x_0,d)
    \quad
    |P_V^\bot(y)| \le \beta(x_0,d) d  \,.
  \end{equation*}
  The vertices of $\mathbf{T}$ lie in $\Sigma \cap \Ball(x_0,d)$ and
  $\mathbf{T}$ is convex, so we also have
  \begin{equation*}
    \forall t \in \mathbf{T}
    \quad
    |P_V^\bot(t)| \le \beta(x_0,d) d \,.
  \end{equation*}
  Let $y \in \mathbf{T}+\mathbf{S}$ and let $t \in \mathbf{T}$ and $s \in
  \mathbf{S}$ be such that $s+t = y$. Using the triangle inequality we see
  that
  \begin{equation*}
    |P_V(y)| \le |y| \le (1 + \beta(0,d)) d  
  \end{equation*}
  and 
  \begin{equation*}
    |P_V^\bot(y)| \le |P_V^\bot(t)|  + |P_V^\bot(s)| \le 2 \beta(x_0,d) d \,.
  \end{equation*}
  Hence, $\mathbf{T}+\mathbf{S}$ is a subset of
  \begin{equation*}
    Z = \big\{ y \in \R^n : |P_V (y)| \le 2 d,\, |P_V^\bot(y)| \le 2 \beta(0,d) d \big\} \,.
  \end{equation*}
  Using again Fubini's theorem, we obtain
  \begin{equation}
    \label{eq:EstimateST}
    \HM^n(\mathbf{T}+\mathbf{S}) \le \HM^n(Z) = C 2^{n-m} \beta(0,d)^{n-m} d^n \,.
  \end{equation}
  Combining \eqref{eq:VolumeST} and \eqref{eq:EstimateST} we finally deduce
  \begin{equation*}
    \HM^{m+1} (T) \le C \beta(0,d) d^{m+1} \,.
  \end{equation*}
\end{proof}

\begin{proof}[Proof of Lemma~\ref{lem:patches}] 
  Fix some $\rho > 0$. Since $\Sigma$ is compact, we can cover it by a finite
  number of balls of radius $\rho$
  \begin{displaymath}
    \Sigma \subseteq \bigcup_{i=1}^N \Ball(x_i,\rho) \,,
    \quad \text{where }x_i \in \Sigma \text{ for } i = 1,\ldots,N \,.
  \end{displaymath}
  This covering has its Lebesgue number, say $\lambda \in (0,\rho)$, so that any
  set of points in $\Sigma$ of diameter less than $\lambda$ lies entirely in one
  of the balls $\Ball(x_i,\rho)$ for some $i \in \{ 1, \ldots, N \}$. Observe
  that if the diameter $\diam(T_k) \ge \lambda$, then $\DC_k(T_k) \le C(n,m)
  \lambda^{-1}$. Also, if $w_0,\ldots,w_{k-1} \in \Sigma_{x_i}^{\rho}$ and
  $w_k,\ldots,w_{m+1} \in \Sigma \setminus \Sigma_{x_i}^{2\rho}$, then the
  diameter $\diam(w_0,\ldots,w_{m+1}) \ge \rho$ and we have
  $\DC(w_0,\ldots,w_{m+1}) \le C(n,m) \rho^{-1}$. Hence, for $T_k \in
  (\Sigma_{x_i}^{\rho})^k$, we have
  \begin{multline*}
    \DC_k(T_k) \le
    \sup_{w_k,\ldots,w_{m+1} \in \Sigma_{x_i}^{2\rho}} \DC(w_0,\ldots,w_{m+1}) \\
    + \sup_{w_k,\ldots,w_{m+1} \in \Sigma \setminus \Sigma_{x_i}^{2\rho}} \DC(w_0,\ldots,w_{m+1}) 
    \le \DC_{k,x_i,2\rho}(T_k) + C(n,m) \rho^{-1} \,.
  \end{multline*}
  In consequence, the following estimate holds
  \begin{multline*}
    \E_p^k(\Sigma) =
    \int_{\Sigma^{k}_{\ge \lambda}} \DC_k(T_k)^p\ d\mu_k(T_k)
    + \int_{\Sigma^{k}_{< \lambda}} \DC_k(T_k)^p\ d\mu_k(T_k) \\
    \le \widetilde{C}(n,m) \HM^m(\Sigma)^{k} (\lambda^{-p} + \rho^{-p})
    + \sum_{i=1}^N \int_{(\Sigma_{x_i}^{\rho})^k_{< \lambda}} \DC_{k,x_i,2\rho}(T_k)^p\ d\mu_k(T_k) \,.
  \end{multline*}
\end{proof}

\section{Regularizing effects of $\E_p^k$}

Let us now prove the other implication of the main theorem, i.e.
\begin{thm}
  \label{thm:RegularizingEffect}
  Let $ k \in \{2, \ldots, m+2\}$ and $p > m(k-1)$ and $\Sigma$ be an
  $m$-dimensional embedded $C^1$ submanifold of the Euclidean space $\R^n$. If
  $\E^k_p (\Sigma)$ is finite, then $\Sigma$ is locally given by graphs functions in 
  $W^{1+s,p}(\mathbb R^m, \mathbb R^{n-m})$,
  where $s=1-\frac{m(k-1)}{p}$.
\end{thm}

Before we start, let us recall a definition of the outer product:

\begin{defin}
  \label{def:wedge}
  Let $w_1$, \ldots, $w_l$ be some vectors in $\R^n$. We define \emph{the outer
    product $w_1 \wedge \cdots \wedge w_l$} to be a vector in $\R^{\binom nl}$,
  whose coordinates are exactly the $l$-minors of the $(l \times n)$-matrix
  $(w_1,\ldots,w_l)$. The coordinates of $w_1 \wedge \cdots \wedge w_l$ are
  indexed by $l$-tuples $(i_1,\ldots,i_l)$, where $i_j \in \{1,\ldots,n\}$ for
  each $j = 1,\ldots,l$ and $i_1 < i_2 < \cdots < i_l$.
\end{defin}

\begin{rem}
  \label{rem:wedge-meas}
  A~standard fact from linear algebra says that the length $|w_1 \wedge \cdots
  \wedge w_l|$ of an outer product of $w_1$, \ldots, $w_l$ is equal to the
  $l$-dimensional volume of the parallelotope spanned by $w_1$, \ldots, $w_l$.
\end{rem}

Now we are ready to prove Theorem~\ref{thm:RegularizingEffect}:

\begin{proof}[Proof of Theorem~\ref{thm:RegularizingEffect}]
  For a point $p \in \Sigma$ we have to show that a small neighborhood
  of $p$ in $\Sigma$ can be given as the graph of a $W^{1+s,p}$ function on 
  $\mathbb R^m$. We can assume after a suitable translation and rotation that $p=0$
  and since $\Sigma$ is of class $C^1$ that there is a function
  $f \in C^1(\R^m, \R^{n-m})$ satisfying $f(0)=0$,
  \begin{equation*}
    \|Df\|_{L^\infty} \le 1
    \quad \text{and} \quad
    g(\Disc_{2 \delta}) \subset \Sigma,
  \end{equation*}
  where $g(x) = (x,f(x))$.
  We will show that then $f \in B^{1+s,p}(\Disc_{\delta})$. Using that 
  $B^{1+s,p}(\Disc_{\delta})=W^{1+s,p}(\Disc_{\delta})$ by Theorem~\ref{thm:equiv-norms}
  and the extension Theorem~\ref{thm:extension} this proves Theorem~\ref{thm:RegularizingEffect}.

  Recalling Definition~\ref{def:wedge} and Remark~\ref{rem:wedge-meas}, for $y,
  w_1, \ldots, w_{m+1} \in \Disc_{\delta}$ the following holds
  \begin{multline}
    \label{eq:simp-meas}
    \HM^{m+1} (\simp(g(y), g(y+w_1), \ldots, g(y+w_{m+1}))) \\
    = \frac 1{m+1} \left| (g(y + w_1)-g(y)) \wedge \ldots \wedge (g(y + w_{m+1})- g(y))\right| \\
    = \frac 1 {m+1} \left| 
      \binom{f(y+w_1) -  f(y)}{w_1}
      \wedge \ldots \wedge 
      \binom{f(y+w_{m+1}) -  f(y)}{w_{m+1}}
    \right| \,.
  \end{multline}
  For fixed $w_1 \in \R^n$ let us set
  \begin{align*}
    \Omega_{w_1}^k := \Big\{(w_2, \ldots, w_{k-1}) \in (\Disc_{\delta})^{k-2}):
    &|w_i| \le |w_1|\ \forall i \in \{2, \ldots, k-1\} \,,\\
    &|w_2 \wedge \ldots \wedge w_{k-1}| \ge \frac 1 2 |w_1|^{k-2} 
    \Big\} \,.
  \end{align*}
  An easy scaling argument leads to 
  \begin{equation}
    \label{est:omega-meas}
    \HM^{m(k-2)}(\Omega_{w_1}^k) 
    = |w_1|^{m(k-2)}\HM ^{m(k-2)}\big(\Omega_{\frac{w_1}{|w_1}|}^k\big)
    = c |w_1|^{m(k-2)} \,,
    \quad \text{where } c=\HM ^{m(k-2)}\big(\Omega_{\frac{w_1}{|w_1|}}^k\big)
  \end{equation}
  obviously does not depend on $w_1$.

  \begin{rem}
    Please note that all the following estimates also hold for $k=m+2$ using the
    convention that there is no supremum and $m+2$ integrals in this case.
  \end{rem}

  Using~\eqref{eq:simp-meas} we can write
  \begin{multline*}
    \E^k_p (\Sigma) 
    = \int_{\Sigma^k} \sup_{x_k, \ldots, x_{m+1} \in \Sigma} \DC(x_0, \ldots, x_{m+1})^p
    \ d\HM^{mk}_{x_0,\ldots,x_{k-1}} \\
    \ge c \int_{(\Disc_\delta)^k} \sup_{\substack{w_j \in \Disc_{\delta}\\j=k,\ldots,m+1}} 
    \frac{\HM^{m+1}(\simp(g(y), g(y+w_1), \ldots, g(y+w_{m+1})))^p}{\diam (T)^{p(m+2)}}   
    \ dw_{k-1} \cdots dw_1\ dy \\
    \ge \bar{c}
    \int_{(\Disc_\delta)^2} \int_{\Omega_{w_1}^k} |w_1|^{-p(m+2)} 
    \sup_{\substack{w_j \in \Disc_{\delta}\\j=k,\ldots,m+1}} 
    \left|
      \binom{f(y+w_1) -  f(y)}{w_1}
      \wedge \cdots \wedge
      \binom{f(y+w_{m+1}) -  f(y)}{w_{m+1}}
    \right|^p\\
    \ dw_{k-1} \cdots dw_1\ dy  \,.
  \end{multline*}
  Now, we use a simple trick: we write the last line as $\tilde c/2$ times twice
  the integral. We leave the first as it is and substitute $w_1 \mapsto -w_1$ in
  the second integral to get
  \begin{multline*}
    \E^k_p (\Sigma) 
    \ge \frac{\bar{c}}2
    \Bigg\{
    \int_{(\Disc_\delta)^2} \int_{\Omega_{w_1}^k} |w_1|^{-p(m+2)} 
    \sup_{\substack{w_j \in \Disc_{\delta}\\j=k,\ldots,m+1}} 
    \left| \binom{f(y+w_1) -  f(y)}{w_1} \wedge \cdots \right. \\
    \left. \cdots \wedge \binom{f(y+w_{m+1}) -  f(y)}{w_{m+1}} \right|^p
    \ dw_{k-1} \cdots dw_1\ dy\\
    + \int_{(\Disc_\delta)^2} \int_{\Omega_{w_1}^k} |w_1|^{-p(m+2)} 
    \sup_{\substack{w_j \in \Disc_{\delta}\\j=k,\ldots,m+1}} 
    \left| \binom{f(y-w_1) -  f(y)}{-w_1} \wedge \cdots \right. \\
    \left. \cdots \wedge \binom{f(y+w_{m+1}) -  f(y)}{w_{m+1}} \right|^p
    \ dw_{k-1} \cdots dw_1\ dy \Bigg\} \,.
  \end{multline*}
  Next, we apply the triangle inequality for the supremum norm obtaining
  \begin{multline}
    \label{est:Epk-triangle}
    \E^k_p (\Sigma) 
    \ge \frac{\bar{c}}2
    \int_{(\Disc_\delta)^2} \int_{\Omega_{w_1}^k} |w_1|^{-p(m+2)} 
    \sup_{\substack{w_j \in \Disc_{\delta}\\j=k,\ldots,m+1}} 
    \left| \binom{f(y+w_1) -  2f(y) + f(y-w_1)}{0} \wedge \cdots \right. \\
    \left. \cdots \wedge \binom{f(y+w_{m+1}) -  f(y)}{w_{m+1}} \right|^p
    \ dw_{k-1} \cdots dw_1\ dy \,.
  \end{multline}
  To estimate this further, for a given $w_1 \in \R^n$ and $(w_2, \ldots,
  w_{k-1}) \in \Omega^k_{w_1}$, we choose vectors $w_{k}, \ldots, w_{m+1}$ such
  that $w_{k}/|w_1|, \ldots, w_{m+1}/|w_1|$ forms an orthonormal basis of the
  orthogonal complement of $\lin(w_2, \ldots, w_{k-1})$. For $k=1,\ldots,n$.
  Furthermore, we let $e \in \R^{n-m}$ be a unit vector satisfying $\langle
  f(y+w_1)-2f(y)+f(y-w_1), e \rangle = |f(y+w_1)-2f(y)+f(y-w_1)|$ and we set $X
  = \lin\{(e,0), (0,w_2), \ldots, (0,w_{m+1})\} \subseteq \R^n$. For brevity of
  notation we set
  \begin{displaymath}
    v = f(y+w_1)-2f(y)+f(y-w_1) \in \R^{n-m} \,.
  \end{displaymath}
  Observe that the orthogonal projection onto $X$ cannot increase the
  $(m+1)$-dimensional measure of any set. Employing the fact that $(e,0)$ is
  orthogonal to each of $(0,w_k)$ for $k=2,\ldots,m+1$ and then using Laplace
  expansion of the determinant with respect to the first column, we
  obtain\footnote{Note that in the first line the wedged vectors are
    $n$-dimensional while in the second line they are $(m+1)$-dimensional.}
  \begin{align*}
    &\left|
      \binom v0 \wedge
      \binom{f(y+w_2) - f(y)}{0}
      \wedge \ldots \wedge 
      \binom{f(y+w_{m+1}) -  f(y)}{w_{m+1}}
    \right| \\
    &\geq 
    \left|
      \binom{\langle v, e \rangle}{0}
      \wedge \binom{\langle f(y+w_{2}) -  f(y), e \rangle}{w_{2}}
      \wedge \ldots \wedge 
      \binom{\langle f(y+w_{m+1}) -  f(y), e \rangle}{w_{m+1}}
    \right| 
    \\
    &= \left| \det
      \begin{pmatrix}
        \langle v, e \rangle 
        & \langle f(y+w_2) - f(y), e \rangle
        &\cdots
        & \langle f(y+w_{m+1}) - f(y), e \rangle \\
		0 & w_2 & \cdots & w_{m+1}
      \end{pmatrix}
    \right|
    \\
    &=  |f(y+w_1) - 2f(y) + f(y - w_1)| |w_2 \wedge \ldots \wedge w_{k-1}| |w_1|^{m+2-k} 
    \\
    &\geq \frac 12 |f(y+w_1) - 2f(y) + f(y - w_1)| |w_1|^{m}\,.
  \end{align*}
  Hence, for all $w_1 \in \R^m$ and $(w_2, \ldots, w_{k-1}) \in
  \Omega^k_{w_1}$ we have
  \begin{multline}
    \label{est:sup-wedge}
    \sup_{w_k, \ldots, w_{m+1} \in \Disc_{|w_1|}}
    \left|
      \binom{f(y+w_1) -  2f(y) + f(y - w_1)}{0}
      \wedge \cdots \wedge 
      \binom{f(y+w_{m+1}) -  f(y)}{w_{m+1}}
    \right| \\
    \ge  \frac 12 |f(y+w_1) -  2f(y) + f(y - w_1)| |w_1|^{m} \,.
  \end{multline}
  Plugging~\eqref{est:sup-wedge} into~\eqref{est:Epk-triangle}, we finally get
  \begin{multline*}
    \E^k_p (\Sigma)
    \ge c \int_{(\Disc_\delta)^2} \int_{\Omega_{w_1}^k} |w_1|^{-p(m+2)} 
    \sup_{w_k, \ldots, w_{m+1} \in \Disc_{|w_1|}}
    \left| \binom{f(y+w_1) -  2f(y) + f(y - w_1)}{0} \wedge \cdots \right. \\ 
    \left. \cdots \wedge \binom{f(y+w_{m+1}) -  f(y)}{w_{m+1}} \right|^p
    \ dw_{k-1} \cdots dw_1\ dy \\
    \ge \bar{c} \int_{(\Disc_\delta)^2} \int_{\Omega_{w_1}^k} 
    \frac {|f(y + w_1) - 2 f(y) + f(y - w_1)|^p }{|w_1|^{p(m+2)-pm}}
    \ dw_{k-1} \cdots dw_1\ dy\\
    = \tilde{c} \int_{(\Disc_\delta)^2}
    \frac{|f(y + w_1) - 2 f(y) + f(y - w_1)|^p }{|w_1|^{2p-m(k-2)}}\ dw_1\ dy \,.
  \end{multline*}
  By Theorem~\ref{thm:equiv-norms}, we thus have $f \in W^{\tilde s,p}(\Disc_{\delta})$,
  where $\tilde s$ is given through the relation $m+\tilde sp = 2p - m(k-2)$ and hence $\tilde s = 2
  - \frac {m(k-1)}{p} = 1+s$.
\end{proof}

\section*{Acknowledgements}
The first author was supported by Swiss National Science Foun\-dation Grant
Nr. 200020 125127 and ``The Leverhulme Trust''.\\
The second author was supported by the Polish Ministry of Science grant
no.~N~N201 397737 (years 2009-2012).

\bibliography{mencurv}{}
\bibliographystyle{hplain}

\end{document}

%% file: preambule.tex
\usepackage{a4wide}
\usepackage[T1]{fontenc}
\usepackage{amsmath}
\usepackage{amsfonts}
\usepackage{amssymb}
\usepackage{amsthm}
\usepackage{esint}
\usepackage{graphicx}
\usepackage{enumerate}
\usepackage{mathrsfs}
\usepackage{cite}
\usepackage{wasysym}
\usepackage{xr}
\usepackage[footnotesize]{caption}
\usepackage[colorlinks=true, pdfstartview=FitV, linkcolor=blue,
citecolor=blue, urlcolor=blue]{hyperref}
\usepackage{color}

\newtheorem{thm}{Theorem}[section]
\newtheorem*{thm*}{Theorem}

\newtheorem{lem}[thm]{Lemma}
\newtheorem{cor}[thm]{Corollary}

\theoremstyle{definition}
\newtheorem{defin}[thm]{Definition}
\newtheorem{rem}[thm]{Remark}


%% file: macros.tex
\DeclareMathOperator{\dist}{dist}
\DeclareMathOperator{\diam}{diam}
\DeclareMathOperator{\lin}{span}

\DeclareMathOperator{\simp}{\Delta}

\DeclareMathOperator{\graph}{graph}

\newcommand{\E}{\mathcal{E}}
\newcommand{\M}{\mathcal{M}}

\newcommand{\N}{\mathbb{N}}
\newcommand{\Z}{\mathbb{Z}}

\newcommand{\Ball}{\mathbb B^n}

\newcommand{\Disc}{\mathbb B^m}

\newcommand{\R}{\mathbb{R}}
\newcommand{\HM}{\mathcal{H}}

\newcommand{\DC}{\mathcal{K}}
